\documentclass[a4paper,12pt]{amsart}

\usepackage{amssymb}
\usepackage{amsmath,amsthm}
\usepackage{enumerate,paralist}
\usepackage{amstext}
\usepackage{dsfont}
\usepackage[english]{babel}
\usepackage{color}
\usepackage{textgreek}
\usepackage{tikz}
\usetikzlibrary{calc}
\usepackage[section]{placeins}

\usepackage[a4paper, left=2.5cm, right=2.5cm, top=3cm, bottom=3cm]{geometry}
\usepackage{mathtools}
\usepackage{latexsym}
\usepackage{mathrsfs}
\usepackage{MnSymbol}
\usepackage{bbm}
\theoremstyle{plain}
\newtheorem{theorem}{Theorem}

\newtheorem{lemma}{Lemma}

\theoremstyle{definition}

\newtheorem{example}{Example}
\newtheorem{remark}{Remark}
\newtheorem{assumption}{Assumption}

\numberwithin{theorem}{section}
\numberwithin{corollary}{section}
\numberwithin{lemma}{section}
\numberwithin{definition}{section}
\numberwithin{example}{section}
\numberwithin{remark}{section}
\numberwithin{proposition}{section}
\numberwithin{assumption}{section}

\usepackage{color}
\definecolor{MY}{rgb}{0.5,0,0.45}

\newcommand{\1}{\mathds 1}

\newcommand{\eps}{\varepsilon}

\newcommand{\dD}{\mathcal{D}}
\newcommand{\eE}{\mathcal{E}}

\newcommand{\hH}{\mathcal{H}}

\def\al{\alpha}
\def\be{\beta}
\def\la{\lambda}

\def\ep{\epsilon}

\def\si{\sigma}

\def\eps{\varepsilon}

\def\Om{\Omega}
\def\dOm{\partial \Omega}

\def\lp{\left(}
\def\rp{\right)}

\begin{document}

\title[]{A Cheeger inequality for the drift Laplacian with Wentzell boundary condition}

\author{Marie Bormann}
\address{Universität Bonn, Hausdorff Center for Mathematics, Endenicher Allee 60, 53115 Bonn, Germany}
\email{mbormann@uni-bonn.de}
%\address{Universit\"at Leipzig, Fakult\"at f\"ur Mathematik und Informatik, Augustusplatz 10, 04109 Leipzig, Germany and Max Planck Institute for Mathematics in the Sciences, 04103 Leipzig, Germany}
%\email{bormann@math.uni-leipzig.de}

\begin{abstract}
We prove lower bounds for the first non-trivial eigenvalue of the drift Laplacian on manifolds with Wentzell-type boundary condition in terms of some Cheeger-type constants for bulk-boundary interactions. Our results are in the spirit of Cheeger's classical inequality.
\end{abstract}

\date{\today}

\maketitle

Let $\Om$ be a smooth compact connected Riemannian manifold of dimension $d \ge 2$ with smooth boundary $\dOm$. We consider the following eigenvalue problem on $\Om$
\begin{equation}\label{eq:evprob}
	\begin{cases}
		div(\al\nabla f) = - \la \al f &\text{ in } \Om^\circ,\\[0.1cm]
		-\al \partial_N f +\delta div^\tau(\be \nabla^\tau f)= - \la \be f &\text{ on }\dOm,
	\end{cases}
\end{equation}
where $\al$ and $\be$ are appropriate weight functions, $\delta\ge0$, $div^\tau,\nabla^\tau$ denote divergence and gradient on $\dOm$ and $\partial_N$ is the derivative in the direction of the outward-pointing unit normal vector field at the boundary. We prove lower bounds for the first non-trivial eigenvalue in the spirit of the classical Cheeger inequality. We also show that suitable heat kernel estimates imply upper bounds in the spirit of the classical Buser inequality.\\ 
Note that in this eigenvalue problem the eigenvalue appears in both equations. Such problems have been considered e.g.\ in~\cite{binding, below, gilles, buoso, barbu}. For sufficiently regular solutions combining the two equations in~\eqref{eq:evprob} at the boundary results in
\begin{equation*}
	\frac{div(\al\nabla f)}{\al} + \frac{\al \partial_N f}{\be} - \frac{\delta div^\tau(\be \nabla^\tau f)}{\be} =0,
\end{equation*}
i.e.\ a Wentzell-type boundary condition. These are known to arise when allowing for dynamical boundary behaviour, cf.~\cite{MR2215623}.

The left-hand side in \eqref{eq:evprob} is the infinitesimal generator of a Markov process that we refer to as doubly weighted Brownian motion on $\Om$ with sticky reflection from $\dOm$ and with ($\delta>0$) or without ($\delta=0$) weighted Brownian diffusion along the boundary. We rely on the construction of the process in~\cite{gv} where it is defined via the corresponding Dirichlet form. Furthermore, also the associated Fokker-Planck equation has been of interest recently, see~\cite{MR4901547,casteras2025largedeviationsstickyreflectingbrownian,bormann2025gradientflownoneheat}. For further references on first constructions and previous results concerning this process see e.g.\ the introductions of~\cite{gv} and~\cite{brw}. 
By bounding from below the first non-trivial eigenvalue in \eqref{eq:evprob} we gain information on the speed of convergence to equilibrium of this process. 
Upper bounds for the Poincar\'{e} constant (corresponding to lower bounds for the spectral gap) for this process have previously been proved in~\cite{mvr, brw, bormann2024functionalinequalitiesdoublyweighted} in terms of bounds on Ricci and sectional curvature on $\Om$ as well as on the second fundamental form on $\dOm$. 

For simplicity we restrict the discussion to $\delta\in\{0,1\}$ in the following. We formulate the following assumption on the weight functions $\al,\be$ and assume that it is fulfilled throughout the paper.
\begin{assumption}\label{as:weights1}
	Denote by $\la$ the volume measure on $\Om$ and by $\si$ the Hausdorff measure on $\dOm$.
	Let $\al,\be \in C(\Om)$ be strictly positive such that $\sqrt{\al}\in H^1(\Om)$ and $\sqrt{\be}\in H^{1}(\dOm)$ and such that $\mu=\al\la+\be\si$ is a probability measure.
\end{assumption}

\section{Connections with classical results}

We first recall the Dirichlet, Neumann and Steklov eigenvalue problems, for which Cheeger-type inequalities are known. For consistency with the problem we want to study we also include weight functions here. 
Indeed for weighted Dirichlet or Neumann Laplacian and doubly weighted Dirichlet-to-Neumann operator it is known (see e.g.~\cite{setti,MR3310527,jammes}) that the respective spectrum consists of a countable sequence of nonnegative (positive for the Dirichlet case) eigenvalues with no finite accumulation point. The respective eigenvalue problems are
\begin{align*}
	\begin{cases}
		div(\al \nabla f) = -\lambda^D \al f &\text{in } \Om^\circ \\
		f = 0 &\text{on } \dOm
	\end{cases},
	\hspace{0.2cm}
	\begin{cases}
		div(\al \nabla f) = -\lambda^N \al f &\text{in } \Om^\circ\\
		\partial_N f = 0 &\text{on } \dOm
	\end{cases},
	\hspace{0.2cm}
	\begin{cases}
		div(\al \nabla f) = 0 &\text{in } \Om^\circ\\
		\al \partial_N f = \lambda^S \be f &\text{on } \dOm.
	\end{cases}
\end{align*}
We also state the variational characterisations for these eigenvalues
\begin{align}
	\lambda^D_k = \min_{E\in\eE_0(k+1)}  \max_{0\neq f\in E} \frac{\int_\Om |\nabla f|^2 \al d\la }{\int_\Om f^2 \al d\la}, \label{eq:varcharD} \\ 
	\lambda^N_k = \min_{E\in\eE(k+1)}  \max_{0\neq f\in E} \frac{\int_\Om |\nabla f|^2 \al d\la }{\int_\Om f^2 \al d\la}, \label{eq:varcharN}\\ 
	\lambda^S_k = \min_{E\in\eE(k+1)}  \max_{0\neq f\in E} \frac{\int_\Om |\nabla f|^2 \al d\la }{\int_{\dOm} f^2 \be d\si}, \label{eq:varcharS} 
\end{align}
where $\eE(k)$ is the set of all $k$-dimensional subspaces of $H^{1}(\Om,\al)$ and $\eE_0(k)$ is the set of all $k$-dimensional subspaces of $\{f\in H^{1}(\Om,\al)\ |\ f\vert_{\dOm}=0\}$.\\
In~\cite{cheeger} (see also~\cite{MR3923535, MR4267584} for a more general weighted setting) the Cheeger constant $h_C$ has been introduced and shown to provide a lower bound for the first non-trivial eigenvalue of the Dirichlet or Neumann Laplacian.
More precisely, if $\Om$ is a compact $d$-dimensional weighted Riemannian manifold with boundary $\dOm$ then
\begin{equation}\label{eq:cheeger}
h_C:=\inf_{|M|_\al\le |\Om|_\al/2} \frac{|\partial_I M|_\al}{|M|_\al} \text{ and } \la^N_1 \ge \frac{h_C^2}{4},\ \la^D_0 \ge \frac{h_C^2}{4},
\end{equation}
where the infimum is over smooth bounded domains $M$, $\partial_I M:=\partial M \cap \Om^\circ$ and $|\cdot|_\al$ denotes the $\al$-weighted $d$ resp.\ $(d-1)$-dimensional volume. 
Optimality of the factor $1/4$ has been shown in~\cite{MR478248}.
Furthermore, in~\cite{jammes} a similar constant $h_J$ has been introduced in order to bound from below the first non-trivial Steklov eigenvalue. In particular, for a compact manifold $\Om$ with boundary $\dOm$
\begin{equation}\label{eq:jammes}
h_J := \inf_{|M|_\al\le |\Om|_\al/2} \frac{|\partial_I M|_\al}{|\partial_E M|_\be} \text{ and } \la^S_1 \ge \frac{h_C \cdot h_J}{4},
\end{equation}
where $\partial_E M:=\partial M \cap \dOm$, $|\cdot|_\be$ denotes the $\beta-$weighted $(d-1)$-dim.\ volume and $h_C$ is the $\al$-weighted Cheeger constant as in \eqref{eq:cheeger}. 
In the following sections we define constants $h_B, h_D$ and $h_E$ that provide lower bounds on the first non-trivial eigenvalue in \eqref{eq:evprob} for $\delta=1$ or for $\delta=0$, i.e.\ for doubly weighted Brownian motion with sticky reflection and with or without weighted boundary diffusion. We have defined the above-mentioned Cheeger-type constants as infima with the restriction $|M|_\al\le |\Om|_\al/2$. We will also discuss two possible replacements of this restriction.

\section{Cheeger inequality for $\delta=0$ - Weighted Brownian motion with sticky reflection}

We write the eigenvalue problem~\eqref{eq:evprob} for the case $\delta=0$.
\begin{equation}\label{eq:wbmwsr}
	\begin{cases}
		div(\al\nabla f) = -\la^{SR} \al f &\text{ in } \Om^\circ,\\
		-\al \partial_N f = -\la^{SR} \be f &\text{ on }\dOm.
	\end{cases}
\end{equation}
This is associated to the infinitesimal generator of Brownian motion on $\Om$ with sticky reflection from $\dOm$ (but without boundary diffusion) doubly weighted according to the measure $\mu=\al\la + \be\si$. Note also that the eigenvalue problem~\eqref{eq:wbmwsr} is different from the one for the Laplacian with Robin boundary conditions where the parameter in the boundary condition is fixed in advance and not determined as an eigenvalue. See e.g.~\cite{2123-5972,MR3681148} for an introduction as well as a Cheeger-type inequality for the Robin eigenvalue problem.\\
 
Arguing in a weighted $H^{1}(\Om,\al)\times H^{1}(\dOm,\be)$ framework as in~\cite{gilles} one may obtain the following Lemma.

\begin{lemma}
The eigenvalues of problem \eqref{eq:wbmwsr} form a countably infinite set \\$\{\la^{SR}_{k}\ |\ k\in\mathbb{N}\}\subset \mathbb{R}_+$ without finite accumulation point, and so its elements may be
arranged in an increasing sequence
\begin{equation*}
	0=\la^{SR}_{0} < \la^{SR}_{1} \le \la^{SR}_{2} \le \ldots \text{ with } \lim_{k\to\infty} \la^{SR}_{k} = +\infty,
\end{equation*}
and the variational characterisation of these eigenvalues is
\begin{equation}\label{eq:varchar}
	\la^{SR}_{k} = \min_{E\in\hH_{k+1}} \max_{u\in E\setminus\{0\}} \frac{\int_\Om |\nabla u|^2\al d\la}{\int_\Om u^2\al d\la + \int_{\dOm} u^2 \be d\si },
\end{equation}
where $\hH_k$ denotes the set of $k$-dimensional subspaces of $H^1(\Om,\al)$.	
\end{lemma}

Note in particular that connectedness of $\Om$ ensures that $\lambda_0^{SR}=0$ is a simple eigenvalue with constant functions as eigenfunctions, while otherwise different constant values on different connected components of $\Omega$ would be possible. Connectedness of $\dOm$ is not necessary for this.\footnote{Note that for the Steklov problem connectedness of the boundary allows to bound the first Steklov eigenvalue via the spectral gap of the Laplacian on the boundary and hence via Cheeger constants for the boundary, see e.g.~\cite{MR3864501,MR4054110}. Whether similar results hold for the spectral gap of the Laplacian with Wentzell boundary condition will not be studied in this work.}\\
As before we drop the dependence on $\al,\be$ in our notation for the eigenvalues. From the previous Lemma and equations \eqref{eq:varcharD}-\eqref{eq:varcharS} we may deduce the following Lemma. 

\begin{lemma}[Comparison of eigenvalues]\label{lem:evcomp}
For $k\in\mathbb{N}$
\begin{equation*}
	\la^{SR}_{k} \le \la_{k}^N, \ \la^{SR}_{k}\le \la_{k}^D,\text{ and }  \la^{SR}_{k} \le \la_{k}^S.
\end{equation*}
\end{lemma}
Thus none of the known Cheeger-type inequalities for Neumann, Dirichlet or Steklov spectral gap as stated in \eqref{eq:cheeger} and \eqref{eq:jammes} automatically provide a lower bound on $\la^{SR}_1$. Contrarily any lower bound for $\la^{SR}_1$ will also provide a lower bound for the other three spectral gaps.\\
We now define a further Cheeger-type constant
\begin{equation*}
\overline{h}_B := \inf_{|A|_\al + |\partial_E A|_\be\le 1/2} \frac{|\partial_I A|_\al}{|A|_\al+|\partial_E A|_\be}.
\end{equation*}
The 'correct' restriction in the infimum for all of the Cheeger-type constants considered so far is not obvious. In fact we can consider the condition $|A|_\al\le |\Om|_\al/2$ (we denote the corresponding constants by $h_C,h_J,h_B$), or we may instead consider $|\partial_E A|_\be\le |\dOm|_\be/2$ (we denote the corresponding constants by $\tilde{h}_C,\tilde{h}_J,\tilde{h}_B$), or we may consider $|A|_\al+|\partial_E A|_\be\le 1/2$ (we denote the corresponding constants by $\overline{h}_C,\overline{h}_J,\overline{h}_B$). Note that $\tilde{h}_J$ has been considered before in~\cite{escobar,MR1696453}.\\
In general the constants $h_\cdot,\tilde{h}_\cdot,\overline{h}_\cdot$ differ from each other. One may construct a simple example similar to the usual dumbbell but with one of the dumbbell balls replaced by a square of the same volume and constant weight functions. In this example $h_C,\tilde{h}_C,\overline{h}_C$ resp.\ $h_J,\tilde{h}_J,\overline{h}_J$ clearly differ. See Figure~\ref{fig:dumbbell} for a schematic drawing of this example with indications for the hypersurfaces dividing $\Om$ into two parts corresponding to the different versions of $h_C$ and $h_J$.\\

\begin{figure}\label{fig:dumbbell}
\begin{center}
\begin{tikzpicture}[scale=1.2]
	\draw [rounded corners](-1.8cm,0.1cm)-- (0cm,0.1cm)--(0cm,0.8862269cm)--(1.7724538cm,0.8862269cm)--(1.7724538cm,-0.8862269cm)--(0cm,-0.8862269cm)--(0cm,-0.1cm)--(-1.8cm,-0.1cm);
	\draw [rounded corners](-1cm,0.1cm)--(-2cm,0.1cm)-- (-2.03407cm,0.2588cm);
	\draw [rounded corners](-1cm,-0.1cm)--(-2cm,-0.1cm)-- (-2.03407cm,-0.2588cm);
	\draw ($(-3,0) + (15:1cm)$) arc (15:345:1cm);
	\draw (-1,-0.2)--(-1,0.2);
	\draw (-1.1,0.2) node[anchor=south]{$h_{\cdot}$};
	\draw (-0.6,-0.2)--(-0.6,0.2);
	\draw (-0.5,0.2) node[anchor=south]{$\tilde{h}_{\cdot}$};
	\draw (-0.8,0.2)--(-0.8,-0.2);
	\draw (-0.75,-0.2) node[anchor=north]{$\overline{h}_{\cdot}$};
	%\draw (-3,0)--(-3,1cm);
	%\draw (0.8862269cm,0)--(0.8862269cm,1cm);
	%\draw (0.8862269cm,0)--(1.59322cm,0.707cm);
	%\draw (0.8862269cm,0)--(1.8862269cm,0);
\end{tikzpicture}
\end{center}
\caption{A dumbbell with one side replaced by a square of the same volume as the other side}
	\label{fig:dumbbell}
\end{figure}

For the sake of completeness we state the following Lemma on the relations of these different constants.

\begin{lemma}[Comparison of Cheeger-type constants]\label{lem:ccomp}
\begin{gather*}
\tilde{h}_C \le \overline{h}_C \le h_C, \ h_J \le \overline{h}_J \le \tilde{h}_J,\ h_B \le \overline{h}_B,\ \tilde{h}_B \le \overline{h}_B, \\
h_B \le h_C,\ h_B\le h_J, \ h_B\ge \frac{min(h_C,h_J)}{2}.
\end{gather*}
Analogous relations to the second line also hold for $\overline{h}$ and $\tilde{h}$ everywhere instead of $h$.
Furthermore, all of the 9 Cheeger-type constants discussed here are strictly positive.
\end{lemma}

\begin{proof}
We only show $\tilde{h}_C \le \overline{h}_C$ since all of the other inequalities in the first line of the statement can be shown analogously.\\
Let $(A_n)_{n\in\mathbb{N}}$ be a sequence of subsets of $\Om$ with $|A_n|_\al + |\partial_E A_n|_\be \le 1/2$ for all $n\in\mathbb{N}$ and such that $\lim_{n\to\infty} \frac{|\partial_I A_n|_\al}{|A_n|_\al}=\overline{h}_C$ and construct a second sequence $(B_n)_{n\in\mathbb{N}}$ via
	\begin{equation*}
	B_n:= \begin{cases} A_n &\text{ if } |\partial_E A_n|_\be\le|\dOm|_\be/2,\\
		\Om\setminus A_n &\text{else}. 
	\end{cases}
	\end{equation*}
Then $B_n$ fulfills $|\partial_E B_n|_\be\le |\dOm|_\be/2$ and $|\partial_I B_n|_\al=|\partial_I A_n|_\al$. Also either $B_n=A_n$ or 
\begin{equation*}
|A_n|_\al \le \frac{1}{2} - |\partial_E A_n|_\be < \frac{1}{2}- \frac{|\dOm|_\be}{2} = \frac{1}{2}- \frac{1-|\Om|_\al}{2} = \frac{|\Om|_\al}{2}
\end{equation*}
and thus $|B_n|_\al \ge |A_n|_\al$. Therefore
\begin{equation*}
	\frac{|\partial_I B_n|_\al}{|B_n|_\al} \le \frac{|\partial_I A_n|_\al}{|A_n|_\al}\ \forall n\in\mathbb{N}.
\end{equation*}
$\lp \frac{|\partial_I B_n|_\al}{|B_n|_\al} \rp_{n\in\mathbb{N}}$ is a bounded sequence and thus has a converging subsequence (still denoted by the index $n$). It holds
\begin{equation*}
	\lim_{n\to\infty} \frac{|\partial_I B_n|_\al}{|B_n|_\al} \le 	\lim_{n\to\infty} \frac{|\partial_I A_n|_\al}{|A_n|_\al} = \overline{h}_C.
\end{equation*}
Thus $\tilde{h}_C\le \overline{h}_C.$\\
The first two statements in the second line of the statement are obvious and concerning the third statement:
\begin{equation*}
\frac{|\partial_I A|_\al}{|A|_\al+|\partial_E A|_\be} \ge \frac{|\partial_I A|_\al}{2\max(|A|_\al,|\partial_E A|_\be)} \ge  \begin{cases} \frac{h_C}{2} &\text{ if } |A|_\al \ge |\partial_E A|_\be \\ \frac{h_J}{2} &\text{ else}\end{cases} \ge \frac{min(h_C,h_J)}{2}.
\end{equation*}
Finally, positivity of $\tilde{h}_J$ and $h_J$ has been shown in~\cite{jammes} resp.~\cite{MR1696453} and positivity of $h_C$ has been shown in~\cite{cheeger}. To show positivity of $\tilde{h}_C$ let $(A_n)_{n\in\mathbb{N}}$ be a sequence of subsets of $\Om$ with $|\partial_E A_n|_\be \le |\dOm|_\be /2$ for all $n\in\mathbb{N}$ and such that
\begin{equation*}
\lim_{n\to\infty} \frac{|\partial_I A_n|_\al}{|A_n|_\al} = \tilde{h}_C.
\end{equation*}
If there is a subsequence (still denoted by the index $n$) such that $|A_n|_\al\le |\Om|_\al/2$ for all $n\in\mathbb{N}$, then it follows that $\tilde{h}_C\ge h_C>0$. Otherwise we may assume that $|A_n|_\al>|\Om|_\al/2$ for all $n\in\mathbb{N}$, and thus there is a constant $c$ depending on the Riemannian metric and the weights $\al,\be$ such that $|\partial_I A_n|_\al + |\partial_E A_n|_\be\ge c$ for all $n\in\mathbb{N}$, see e.g.~\cite[section 12.3]{MR2455580}. Now for each $n\in\mathbb{N}$ either $|\partial_I A_n|_\al\ge c/2$ which implies 
\begin{equation*}
\frac{|\partial_I A_n|_\al}{|A_n|_\al}>\frac{c}{2|\Om|_\al}>0,
\end{equation*} 
or $|\partial_E A_n|_\be\ge c/2$ which implies 
\begin{equation*}
\frac{|\partial_I A_n|_\al}{|A_n|_\al}=\frac{|\partial_I A_n|_\al}{|\partial_E A_n|_\be}\frac{|\partial_E A_n|_\be}{|A_n|_\al}\ge \tilde{h}_J \frac{c}{2|\Om|_\al}>0.
\end{equation*} 
It follows that $\frac{|\partial_I A_n|_\al}{|A_n|_\al}$ is bounded away from zero uniformly in $n$ and thus $\tilde{h}_C>0$. \\
Positivity of all of the other Cheeger-type constants considered here follows from positivity of $\tilde{h}_C$ and $h_J$ and the relations shown above.
\end{proof}

\begin{remark}
The previous Lemma provides one-sided bounds. In the following we briefly mention some examples using constant weight functions to illustrate that the ratios $\frac{h_C}{\tilde{h}_C},\frac{\overline{h}_B}{\tilde{h}_B}, \frac{\tilde{h}_J}{h_J}, \frac{\overline{h}_B}{h_B}$ may diverge:
\begin{itemize}
\item
Let $\Omega_n:=[0,1]\times S^1$ with metric $g_n:=dx^2 + f_n(x)^2 d\theta^2$ and $f_n(x):= \frac{1+nsin^2(\pi x)}{1+n/2}$. For simplicity we choose constant weights $\al=\overline{\al}/|\Om_n|,\ \be=(1-\overline{\al})/|\dOm_n|$ for some $\overline{\al}\in(0,1)$. That is we consider a cylinder which for growing $n$ shrinks towards the end points 0 and 1 while the radius at $x=1/2$ converges to 2. 
Due to symmetry of $f_n$ optimisers $A_n$ for $h_C$ are of the form $[1/2-a_n,1/2+a_n]\times S^1$ where $a_n$ is constrained by $|A_n|_\al\le |\Omega_n|_\al/2$ which by direct computation and for $n\to\infty$ amounts to $a_\infty+\frac{\sin(2\pi a_\infty)}{2\pi}=\frac{1}{4} \Leftrightarrow a_\infty \approx 0,1324$. Thus $h_C$ can be computed explicitly and can be seen to be positive, while $\tilde{h}_C$ vanishes by choosing $A_n:=[0,1-\epsilon]\times S^1$ for some small $\epsilon>0$.\\
In the same example $\tilde{h}_B$ can be seen to vanish by choosing $A_n:=[0,1-\epsilon]\times S^1$ for some small $\epsilon>0$ while $\overline{h}_B$ stays bounded away from zero by the same arguments as above.
\item
In~\cite{MR4695859,MR1183224} a specific sequence of hyperbolic tori $\Om_\epsilon$ (i.e.\ hyperbolic surfaces of genus 1 with one boundary component) has been constructed by gluing together a hyperbolic cylinder and a hyperbolic torus along a common geodesic boundary component of length $2\pi\epsilon$. We refer to the given references for the details of the construction. It has been shown there that for this sequence $h_J$ converges to zero by choosing $A_\epsilon$ in the definition of $h_J$ to equal the cylindrical part of $\Om_\ep$. However, $\tilde{h}_J$ does not converge to zero: For small $\epsilon$ the hyperbolic cylinder grows longer and thinner while the size of the boundary stays of order 1 and any $A$ with $0<|\partial_E A|<|\partial \Om_\epsilon|/2$ must have an interior boundary whose size is controlled from below.\\
In the same example $h_B$ can be seen to vanish by again choosing $A_\epsilon$ to equal the cylindrical part of $\Om_\ep$ while $\overline{h}_B$ stays bounded away from zero by the same arguments as above.
\end{itemize}
\end{remark}

We now show a lower bound on the spectral gap $\lambda_1^{SR}$ in terms of the product of two Cheeger-type constants. Our proof is similar to~\cite{cheeger,jammes}.

\begin{theorem}[Cheeger-type inequality]\label{thm:cheegerwobd}
For $\mu=\al\la+\be\si$ as introduced above it holds that 
\begin{equation*}
\la^{SR}_1 \ge \frac{\overline{h}_B \cdot \overline{h}_C}{4}.
\end{equation*}
\end{theorem}

\begin{proof}
Let $f$ be an eigenfunction for the first non-trivial eigenvalue $\la^{SR}_1$ and assume that the set $\Om^+:=f^{-1}([0,\infty))$ satisfies $|\Om^+|_\al+|\partial_E \Om^+|_\be\le 1/2$ (else change the sign of $f$). $f$~restricted to $\Om^+$ is then an eigenfunction for the problem 
\begin{equation*}
	\begin{cases}
		div(\al\nabla f) = -\la \al f &\text{ in } (\Om^+)^\circ,\\
		-\al \partial_N f = -\la \be f &\text{ on } \partial_E\Om^+, \\
		f=0 &\text{ on } \partial_I \Om^+,
	\end{cases}
\end{equation*}
with eigenvalue $\la=\la^{SR}_1$ and $\partial_E \Om^+:=\partial \Om^+\cap \partial \Om, \partial_I \Om^+:=\partial \Om^+\cap \Om^\circ$. Thus via integration by parts
\begin{align*}
\la^{SR}_1 = \frac{\int_{\Om^+} |\nabla f|^2\al d\la}{\int_{\Om^+} f^2\al d\la + \int_{\dOm^+} f^2\be d\si} &= \frac{\int_{\Om^+} |\nabla f|^2\al d\la \cdot \int_{\Om^+} f^2\al d\la}{\lp \int_{\Om^+} f^2\al d\la + \int_{\dOm^+} f^2\be d\si\rp \cdot \int_{\Om^+} f^2\al d\la} \\
		&\ge \frac{\lp \int_{\Om^+} |\nabla f^2|\al d\la\rp^2}{4\lp \int_{\Om^+} f^2\al d\la + \int_{\dOm^+} f^2\be d\si\rp \cdot \int_{\Om^+} f^2\al d\la}. 
\end{align*}
We set $D_t:=f^{-1}([\sqrt{t},+\infty))$ and use the Coarea formula to obtain
\begin{align}
\int_{\Om^+} |\nabla f^2| \al d\la &= \int_{t\ge0} |\partial_I D_t|_\al dt, \label{eq:coarea}\\
\int_{\Om^+} f^2 \al d\la &= \int_{t\ge0} |D_t|_\al dt, \\
\int_{\dOm^+} f^2 \be d\si &= \int_{t\ge0} |\partial_E D_t|_\be dt. \label{eq:sublevelsets}
\end{align}
For each $t\ge0$ we have that $|D_t|_\al+ |\partial_E D_t|_\be \le |\Om^+|_\al +  |\partial_E \Om^+|_\be \le 1/2$.
Thus
\begin{align*}
\la^{SR}_1 &\ge \frac{\int_{t\ge0} \int_{s\ge0} |\partial_I D_t|_\al |\partial_I D_s|_\al \ dsdt}{4 \int_{t\ge0} \int_{s\ge0} \lp|D_t|_\al + |\partial_E D_t|_\be\rp |D_s|_\al \ ds dt} \\
	&\ge \frac{\int_{t\ge0} \int_{s\ge0} \overline{h}_B\overline{h}_C  \lp|D_t|_\al + |\partial_E D_t|_\be\rp |D_s|_\al \ dsdt}{4 \int_{t\ge0} \int_{s\ge0} \lp|D_t|_\al + |\partial_E D_t|_\be\rp |D_s|_\al \ ds dt} \\
	&= \frac{\overline{h}_B \overline{h}_C}{4}.
\end{align*}
\end{proof}

\begin{remark}
Note that by replacing the condition $|\Om^+|_\al+|\partial_E \Om^+|_\be\le 1/2$ by $|\Om^+|_\al\le |\Om|_\al/2$ or $|\partial_E \Om^+|_\be\le |\dOm|_\be/2$ we may prove in the same way
\begin{equation*}
\la^{SR}_1 \ge \frac{h_B \cdot h_C}{4}, \ \la^{SR}_1 \ge \frac{\tilde{h}_B \cdot \tilde{h}_C}{4}.
\end{equation*}
From Lemma~\ref{lem:ccomp} it is not obvious whether $\overline{h}_B\overline{h}_C$ or $h_B h_C$ is the better lower bound. The lower bound in \eqref{eq:jammes} also holds with $h$ replaced by $\tilde{h}$ (cf.~\cite{jammes}) or $\overline{h}$.
\end{remark}

From Lemma~\ref{lem:evcomp} it follows that $h_B \cdot h_C/4$ is also a lower bound for $\la_{1}^N, \  \la_{1}^D$ and $\la_{1}^S$ (and via Friedlander's inequality also for $\la^D_0$), however we also see from Lemma~\ref{lem:ccomp} that this lower bound is worse than the respective known lower bounds in \eqref{eq:cheeger} and \eqref{eq:jammes} in terms of $h_C$ and $h_J$. The analogous statement holds if $h$ is replaced with $\tilde{h}$ or $\overline{h}$.\\

It is known that dumbbell manifolds form extremal examples for the classical Cheeger constant in the sense that one may construct a sequence of manifolds of dumbbell shape for which the associated sequence of Cheeger constants converges to zero. As $h_B \le h_C$, such a sequence would also be extremal for our Cheeger-type constant $h_B$.\\
This classical example illustrates that the rate of convergence to equilibrium for reflected Brownian motion in dumbbell domains is very low, as the process can only spread out through the domain rather slowly due to the bottleneck. Adding stickiness to the boundary corresponds to slowing the process down at the boundary but it does not change the behaviour of the process away from the boundary. I.e.\ if reflected Brownian motion spreads out slowly in a particular type of domain, then so does sticky-reflecting Brownian motion. However, the reverse implication is not true in general: In the next remark we consider an example where the effect of the sticky boundary behaviour is so dominant that sticky-reflecting Brownian motion spreads out very slowly in space while reflected Brownian motion fares significantly better.\\

\begin{remark}
One may ask whether one can also prove a lower bound for $\la^{SR}_1$ of the form $c h_C^\iota h_B^\kappa$ for universal constants $c,\iota,\kappa\ge0$ with values better than $c=1/4, \iota=\kappa=1$ as in Theorem~\ref{thm:cheegerwobd}.\footnote{For a study of the relationship between the first eigenvalue of the Laplacian and Cheeger constant on closed manifolds when the Cheeger constant converges to zero see e.g.~\cite{MR2013096}.}\\
We multiply the Riemannian metric by a factor $\eta^2, \eta>0$, which changes the volume measure $\la$ resp.\ Hausdorff measure $\si$ by a prefactor $\eta^d$ resp.\ $\eta^{d-1}$. We accordingly scale the weight functions $\al$ resp.\ $\be$ by factors $\eta^{-d}$ resp.\ $\eta^{-(d-1)}$ to preserve normalisation of the measure $\al\la+\be\si$. Furthermore, this conformal change of metric produces a factor $1/\eta^2$ for the squared norm of the gradient $|\nabla \cdot|^2$, hence from the variational characterisation \eqref{eq:varchar} we find that $\la^{SR}_1$ changes by a factor by $1/\eta^2$. $h_C$ as well as $h_B$ change by a factor $1/\eta$.
From letting $\eta$ tend to $+\infty$ we may now deduce that a lower bound of the proposed form can only hold if $\iota+\kappa\ge 2$, while letting $\eta$ tend to 0 shows that $\iota\le 2$ is necessary.\\
From Lemma~\ref{lem:ccomp} we know that $h_B\le h_C$, thus a better lower bound than the one corresponding to $c=1/4, \iota=\kappa=1$ would mean $\iota>1$ and $\kappa<1$. The following example shows that this is not possible in general.\\
As in Example 4 in~\cite{jammes} we consider $M_n:=M\times[0,\frac{1}{n}]$ where $M$ is some compact smooth Riemannian manifold without boundary. We consider constant weight functions $\al$ and $\be$ here. Then $M_n$ is a smooth compact Riemannian manifold with a 
smooth boundary. By Lemma~\ref{lem:evcomp} it holds
\begin{equation*}
	\la^{SR}_1 \le \la_1^S,
\end{equation*}
where $\la_1^S$ is the first non-trivial Steklov eigenvalue on $M_n$ and as in Lemma 6.1 in~\cite{MR2807105} one may show that $\la_1^S \sim \frac{1}{n}$ as $n\to\infty$. It follows that $\la^{SR}_1$ also goes to zero for $n\to \infty$.\\
Furthermore, for $D\subset M_n$ it holds
\begin{equation*}
	\frac{|\partial_{I}D|_\al}{|D|_\al} >h_C(M\times \mathbb{S}^1) > 0, 
\end{equation*} 
which has been shown in Example 4 of~\cite{jammes}. Thus after taking the infimum we deduce that $h_C(M_n)$ is bounded away from zero uniformly for $n\in\mathbb{N}$. By the classical Cheeger
inequality~\eqref{eq:cheeger} it follows that $\la_1^N$ also stays bounded away from zero.\\
Finally for some $U\subset M$ sufficiently small
\begin{equation*}
	h_B(M_n)\le \frac{|\partial_I(U\times[0,1/n]))|_\al}{|U\times[0,1/n]|_\al+|\partial_E (U\times[0,1/n])|_\be} = \frac{|\partial_I U|}{|U|+n|U|},
\end{equation*}
and thus $h_B(M_n)\sim \frac{1}{n}$ as $n\to\infty$. Summing up we may deduce that $\kappa\ge 1$ is necessary. The same argument is valid for $h$ replaced by $\overline{h}$ or $\tilde{h}$ everywhere. 
\end{remark}

\begin{remark}\label{rem:constweightswobd}
We consider the case where $\al=\bar{\al}/|\Om|$, $\be=(1-\bar{\al})/|\dOm|$ for $\bar{\al}\in(0,1)$ constant, i.e.\ constant weight functions $\al,\be$. This corresponds to the following infinitesimal generator
\begin{equation*}
\bar{L}f := \1_\Om \Delta f - \1_{\dOm} \gamma \partial_N f, 
\end{equation*}
where $\gamma:=\frac{\bar{\al}}{1-\bar{\al}} \frac{|\dOm|}{|\Om|}$. For growing $\gamma$ the effect of the reflection from the boundary becomes stronger, while the effect of the stickiness loses relevance, and the diffusion thus increasingly resembles usual reflected Brownian motion without stickiness.\\
For these constant weight functions the lower bound in Theorem~\ref{thm:cheegerwobd} becomes 
\begin{equation*}
\frac{h_B \cdot h_C}{4} = \inf_{|A|\le |\Om|/2} \frac{|\partial_I A|}{|A|+\frac{1}{\gamma}|\partial_E A|}\cdot \frac{h_C}{4}.
\end{equation*}
Heuristically exchanging infimum and limit this tends to $h_C^2/4$ as $\gamma$ goes to $\infty$, thus reproducing the classical Cheeger inequality in the limit. The same holds for $h$ replaced by $\overline{h}$ or $\tilde{h}$ everywhere.
Note however that we have not rigorously shown $\Gamma$-convergence.
\end{remark}

\begin{example}
In ~\cite{brw} the exact value for the Poincar\'{e} constant/spectral gap for Brownian motion on a closed unit ball in Euclidean/hyperbolic plane with sticky reflection from the boundary (but without boundary diffusion) has been computed numerically in the setting of constant weight functions as introduced in Remark~\ref{rem:constweightswobd}. We may thus evaluate our lower bound for the spectral gap obtained from Theorem~\ref{thm:cheegerwobd} and compare it with the exact values. Note that we do not rigorously show that the infima in
the definitions of the Cheeger-type constants are obtained for our choice of subset of $\Om$.\\
To evaluate our bound for the example of a closed unit ball $\Om$ in the Euclidean plane we note that by choosing $A$ as a half circle 
\begin{equation*}
h_C = \frac{4}{\pi}, \ h_B=\inf_{|A|\le |\Om|/2} \frac{|\partial_I A|}{|A|+\frac{1}{\gamma}|\partial_E A|} \le \frac{4 \overline{\al}}{\pi}.
\end{equation*}
And thus 
\begin{equation*}
\frac{h_B \cdot h_C}{4} \le \frac{4\overline{\al}}{\pi^2}.
\end{equation*}
Analogously for the example of a closed unit ball Ω in the hyperbolic plane note that again by choosing $A$ as a half circle 
\begin{equation*}
h_C = \frac{2}{\pi(\cosh(1)-1)}, \ h_B = \inf_{|A|\le |\Om|/2} \frac{|\partial_I A|}{|A|+\frac{1}{\gamma}|\partial_E A|} \le \frac{2\overline{\al}}{\pi(\cosh(1)-1)}.
\end{equation*}
And thus 
\begin{equation*}
\frac{h_B\cdot h_C}{4} \le \frac{\overline{\al}}{\pi^2(\cosh(1)-1)^2}.
\end{equation*}
As we always choose $A$ to be a half circle for this simple example it does not make a difference whether we consider $h, \overline{h}$ or $\tilde{h}$.\\
We plot the resulting curves for exact values of the spectral gap and lower bounds for varying values of $\overline{\al}\in(0,1)$ for the two examples in Figures~\ref{fig:euclidcheeger} and~\ref{fig:hyperbolcheeger}. In both cases the lower bounds obtained from Theorem~\ref{thm:cheegerwobd} are not very exact. Note however for comparison that for the classical Cheeger inequality~\eqref{eq:cheeger} (again by choosing half-circles for $M$) we have in the case of a unit ball in the Euclidean plane
\begin{equation*}
\lambda_1^N \approx 3.39 \text{ and } \frac{h_C^2}{4} = \frac{4}{\pi^2} \approx 0.41,
\end{equation*}
and in the case of a unit ball in the hyperbolic plane
\begin{equation*}
\lambda_1^N \approx 2.96 \text{ and } \frac{h_C^2}{4} = \frac{1}{\pi^2(\cosh(1)-1)^2} \approx 0.34.
\end{equation*}
\end{example}

\begin{figure}
	\includegraphics[width=0.9\linewidth]{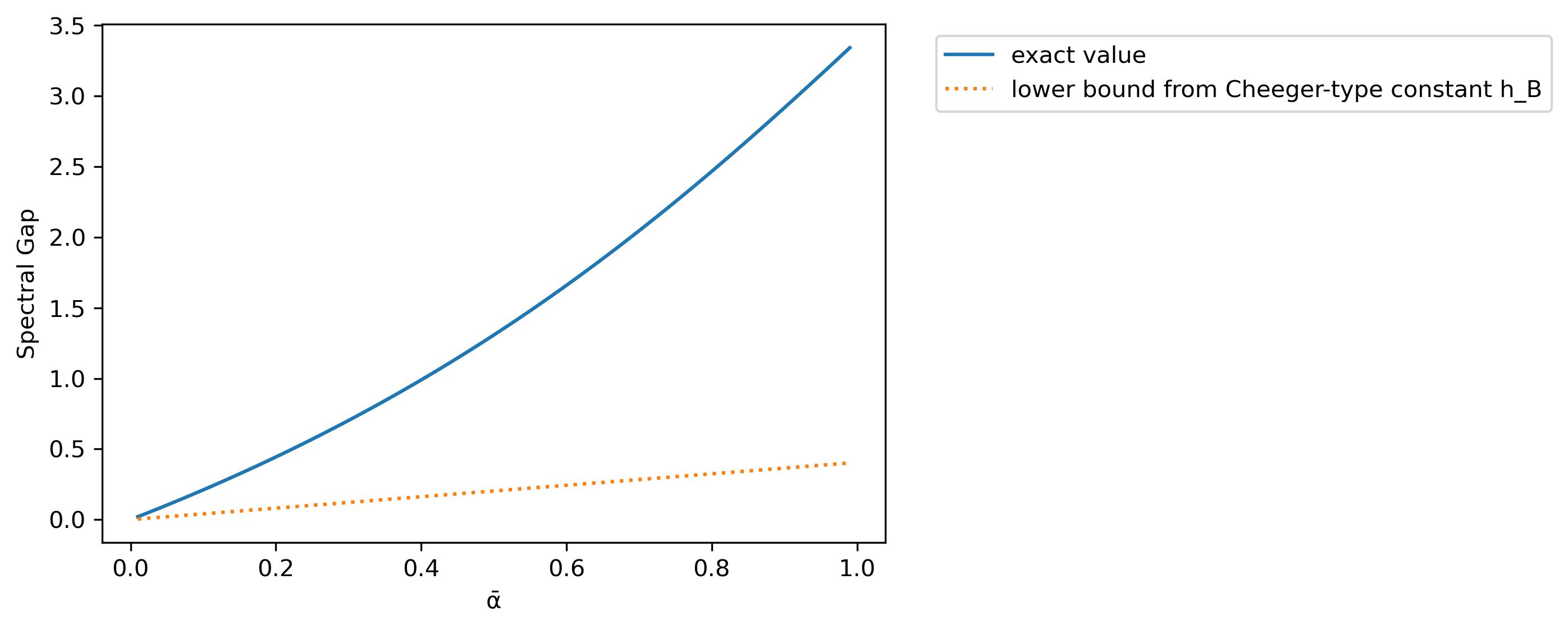}
	\caption{Exact spectral gap (blue, solid) and lower bound (yellow, dotted) for Euclidean example}
	\label{fig:euclidcheeger}
\end{figure}

\begin{figure}
	\includegraphics[width=0.9\linewidth]{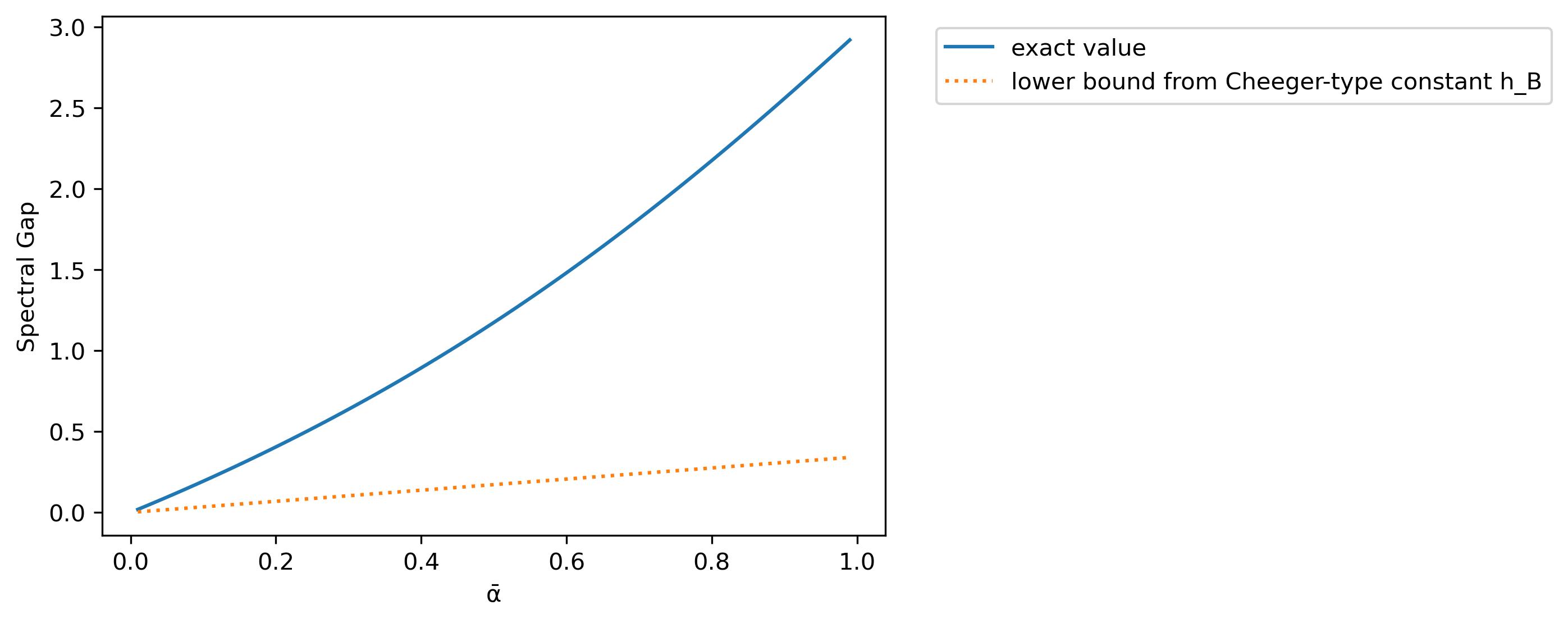}
	\caption{Exact spectral gap (blue, solid) and lower bound (yellow, dotted) for hyperbolic example}
	\label{fig:hyperbolcheeger}
\end{figure}

\section{Cheeger inequality for $\delta=1$ - Weighted Brownian motion with sticky reflection and boundary diffusion}

We write the eigenvalue problem~\eqref{eq:evprob} for the case $\delta=1$.
\begin{equation}\label{eq:evprob2}
	\begin{cases}
		div(\al\nabla f) = - \la^{SRBD} \al f &\text{ in } \Om^\circ,\\
		-\al \partial_N f + div^\tau(\be \nabla^\tau f)= - \la^{SRBD} \be f &\text{ on }\dOm,
	\end{cases}
\end{equation}
This is associated to the infinitesimal generator of Brownian motion on $\Om$ with sticky reflection and boundary diffusion doubly weighted according to $\mu=\al\la + \be\si$.\\
Note that our restriction to the case $\delta=1$ in problem~\eqref{eq:evprob} instead of arbitrary $\delta>0$ is just for simplicity. In fact the case $1\neq\delta>0$ can be treated analogously to the case $\delta=1$, see Remark~\ref{rem:rem2}.\\

Again arguing in a weighted $H^1(\Om,\al)\times H^1(\dOm,\be)$ framework we obtain the following Lemma.

\begin{lemma}
The eigenvalues of problem \eqref{eq:evprob2} form a countably infinite set \\$\{\la^{SRBD}_{k}\ |\ k\in\mathbb{N}\}\subset \mathbb{R}_+$ without finite accumulation point, and so its elements may be
arranged in an increasing sequence
\begin{equation*}
	0=\la^{SRBD}_{0} < \la^{SRBD}_{1} \le \la^{SRBD}_{2} \le \ldots \text{ with } \lim_{k\to\infty} \la^{SRBD}_{k} = +\infty,
\end{equation*}
and the variational characterisation of these eigenvalues is
\begin{equation}\label{eq:varchar2}
	\la^{SRBD}_{k} = \min_{E\in\hH_{k+1}} \max_{u\in E\setminus\{0\}} \frac{\int_\Om |\nabla u|^2\al d\la + \int_{\dOm} |\nabla^\tau u|^2 \be d\si}{\int_\Om u^2\al d\la + \int_{\dOm} u^2 \be d\si },
\end{equation}
where $\hH_k$ denotes the set of $k$-dimensional subspaces of $\hH$ by which we denote the closure of $C^1\lp\overline{\Om}\rp$ under the norm $\|\cdot\|^2:=|\cdot|^2_{H^1(\Om^\circ,\al)} + |\cdot|^2_{H^1(\dOm,\be)}$.
\end{lemma}

Note again that connectedness of $\Om$ ensures that $\lambda_0^{SRBD}=0$ is a simple eigenvalue, while connectedness of $\dOm$ is not necessary.\\
From \eqref{eq:varchar} and \eqref{eq:varchar2} it is obvious that

\begin{lemma}[Comparison of eigenvalues]\label{lem:evcomp2}
For $k\in\mathbb{N}$
\begin{equation*}
	\la^{SRBD}_k\ge \la^{SR}_k.
\end{equation*}
\end{lemma}
\bigskip
We define the following Cheeger-type constant
\begin{equation*}
\overline{h}_D : = \inf_{A,B,C} \frac{|\partial_I A|_\al }{|A|_\al} \frac{ |\partial_I B|_\al}{ |B|_\al + |\partial_E B|_\be}
	+  \frac{|\partial\partial_E B|_\be}{|B|_\al + |\partial_E B|_\be}\frac{|\partial \partial_E C|_\be}{|\partial_E C|_\be},
\end{equation*}
where the infimum is over $A,B,C$ such that $|A|_\al+|\partial_E A|_\be\le 1/2$ and analogously for $B,C$ and where for $B\subset \Om$ by $\partial \partial_E B$ we denote the boundary of $\partial_E B$ in $\dOm$ and analogously for $C$. Again we may define $h_D$ and $\tilde{h}_D$ analogously by exchanging the restrictions for the infimum.\\
Furthermore, we define
\begin{equation*}
\tilde{h}_E := \inf_{|\partial_E B|_\be \le |\dOm|_\be/2}  \frac{ |\partial_I B|_\al + |\partial \partial_E B|_\be}{ |B|_\al + |\partial_E B|_\be}.
\end{equation*}

We now show a lower bound on the spectral gap $\lambda_1^{SRBD}$ in terms of the new Cheeger-type constant $\overline{h}_D$. We also deduce a simpler lower bound in terms of $\tilde{h}_E$ which more clearly illustrates the interplay of bulk and boundary diffusion. Again the proof is similar to~\cite{cheeger,jammes}.

\begin{theorem}[Cheeger-type inequality]\label{thm:cheegerwbd}
For $\mu=\al\la+\be\si$ as introduced above it holds that
\begin{equation}\label{eq:lowerbound1}
\la^{SRBD}_1 \ge \frac{\overline{h}_D}{4},
\end{equation}
and the right-hand side is strictly positive. The same holds with $\overline{h}$ replaced by $\tilde{h}$ or $h$. Furthermore,
\begin{equation}\label{eq:lowerbound2}
\la^{SRBD}_1 \ge \frac{min(\tilde{h}_C(\Om),h_C(\dOm))\cdot \tilde{h}_E}{4},
\end{equation}
where 
\begin{equation*}
\tilde{h}_C(\Om):=\inf_{\substack{A\subset \Om\\ |\partial_E A|_\be \le |\dOm|_\be/2}} \frac{|\partial_I A|_\al}{|A|_\al},\  h_C(\dOm):=\inf_{\substack{C\subset \dOm\\|C|_\be\le|\dOm|_\be/2}}\frac{|\partial C|_\be}{|C|_\be}.
\end{equation*}
If $\dOm$ is not connected, then the lower bound in~\eqref{eq:lowerbound1} reduces to $\overline{h}_D/4=\overline{h}_B\overline{h}_C/4$ (again $\overline{h}$ may be replaced by $\tilde{h}$ or $h$) and the lower bound in~\eqref{eq:lowerbound2} reduces to the trivial lower bound zero.
\end{theorem}

\begin{proof}
The strict positivity of $\overline{h}_D$ follows from Lemma~\ref{lem:ccomp} since $\overline{h}_D\ge \overline{h}_C\cdot\overline{h}_B$.\\
We further proceed analogously to the proof of Theorem~\ref{thm:cheegerwobd}.
Let $f$ be an eigenfunction for the first non-trivial eigenvalue $\la^{SRBD}_1$ and assume that $\Om^+:=f^{-1}([0,\infty))$ satisfies $|\Om^+|_\al + |\partial_E \Om^+|_\be \le 1/2$ (else change the sign of $f$). $f$ restricted to $\Om^+$ is then an eigenfunction for the problem
\begin{equation*}
	\begin{cases}
		div(\al\nabla f) = - \la \al f &\text{ in } (\Om^+)^\circ,\\
		-\al \partial_N f + div^\tau(\be \nabla^\tau f)= - \la \be f &\text{ on }\partial_E\Om^+,\\
		f=0 &\text{ on } \partial_I \Om^+,
	\end{cases}
\end{equation*}
with $\la=\la^{SRBD}_1$ being the first eigenvalue. It then holds
\begin{align*}
\la^{SRBD}_1 &= \frac{\int_{\Om^+} |\nabla f|^2\al d\la + \int_{\dOm^+} |\nabla^\tau f|^2\be d\si}{\int_{\Om^+} f^2\al d\la + \int_{\dOm^+} f^2\be d\si} \\
		&= \frac{\lp \int_{\Om^+} |\nabla f|^2\al d\la + \int_{\dOm^+} |\nabla^\tau f|^2\be d\si\rp \cdot \int_{\Om^+} f^2\al d\la \cdot \int_{\dOm^+} f^2\be d\si}{\lp \int_{\Om^+} f^2\al d\la + \int_{\dOm^+} f^2\be d\si\rp \cdot \int_{\Om^+} f^2\al d\la \cdot \int_{\dOm^+} f^2\be d\si} \\
		&\ge \frac{1}{4} \frac{\lp \int_{\Om^+} |\nabla f^2|\al d\la\rp^2\int_{\dOm^+} f^2\be d\si + \int_{\Om^+}f^2\al d\la \lp \int_{\dOm^+} |\nabla^\tau f^2|\be d\si\rp^2 }{\lp \int_{\Om^+} f^2\al d\la\rp^2\int_{\dOm^+} f^2\be d\si  + \int_{\Om^+} f^2\al d\la \lp\int_{\dOm^+} f^2\be d\si\rp^2} . 
\end{align*}
We set $D_t:=f^{-1}([\sqrt{t},+\infty))$ and use the Coarea formula to obtain
\begin{equation*}
\int_{\dOm^+} |\nabla^\tau f^2| \be d\si = \int_{t\ge0} |\partial \partial_E D_t|_\be dt.\\
\end{equation*}
Recall also equations~\eqref{eq:coarea}-\eqref{eq:sublevelsets} and note that $|D_t|_\al + |\partial_E D_t|_\be \le |\Om^+|_\al + |\partial_E \Om^+|_\be \le 1/2$ for each $t\ge0$.
Thus
\begin{align*}
\la^{SRBD}_1 &\ge \frac{1}{4}\frac{\int_{t\ge0} \int_{s\ge0} \int_{r\ge0} |\partial_I D_t|_\al |\partial_I D_s|_\al |\partial_E D_r|_\be + |D_t|_\al|\partial\partial_E D_s|_\be|\partial \partial_E D_r|_\be\ dr ds dt}{\int_{t\ge0} \int_{s\ge0} \int_{r\ge0} |D_t|_\al |D_s|_\al |\partial_E D_r|_\be + |D_t|_\al|\partial_E D_s|_\be|\partial_E D_r|_\be \ dr ds dt} \\
	&\ge \frac{1}{4}\frac{\int_{t\ge0} \int_{s\ge0} \int_{r\ge0} \overline{h}_D \lp   |D_t|_\al |D_s|_\al |\partial_E D_r|_\be + |D_t|_\al|\partial_E D_s|_\be|\partial_E D_r|_\be \rp\ dr ds dt}{\int_{t\ge0} \int_{s\ge0} \int_{r\ge0}  |D_t|_\al |D_s|_\al |\partial_E D_r|_\be + |D_t|_\al|\partial_E D_s|_\be|\partial_E D_r|_\be \ dr ds dt} \\
	&\ge \frac{\overline{h}_D}{4}.
\end{align*}
The statement in terms of $\tilde{h}_D$ or $h_D$ can be shown analogously by changing the assumption on $ \Om^+$. The second statement of the Proposition follows directly from $\la^{SRBD}_1 \ge \frac{\tilde{h}_D}{4}$ and the definition of $\tilde{h}_D$ and $\tilde{h}_E$. Again the right-hand side is strictly positive due to $\tilde{h}_E\ge\tilde{h}_B$ and Lemma~\ref{lem:ccomp}.\\
The last statement concerning the case of non-connected $\dOm$ follows by choosing in the definition of $\overline{h}_D$ the set $C$ to be a sufficiently small $\epsilon$-neighbourhood of a connected component $G$ of $\dOm$ with $|G|_\be\le |\dOm|_\be/2$. By the same argument $h_C(\dOm)$ can be seen to vanish in this case.
\end{proof}

\begin{remark} Unlike the lower bound in Theorem~\ref{thm:cheegerwobd} it is not possible to rewrite $h_D$ as a product of more simple infima. Note also that by taking $C$ such that $\partial_E C=\dOm$ and correspondingly $\partial \partial_E C =\emptyset$ (which can be done e.g.\ by choosing $C$ as an $\eps$-neighbourhood of the boundary)
\begin{equation*}
	h_D = h_B\cdot h_C
\end{equation*} 
would result in $h_B\cdot h_C/4$ as a lower bound for $\la^{SRBD}_1$. This seems suboptimal since we know $\la_1^{SRBD}\ge\la_1^{SR}$ from Lemma~\ref{lem:evcomp2} and would therefore prefer to not obtain the same lower bound on $\la^{SR}_1$ and $\la_1^{SRBD}$. Choosing $C$ in this way for $\overline{h}_D$ instead of $h_D$ is only possible if the weight function $\be$ is small enough. It is not possible, however, if we consider $\tilde{h}_D$. I.e.\ we get a lower bound for $\la^{SRBD}_1$ in terms of $\tilde{h}_D$ that is in general not equal to the corresponding lower bound for $\la^{SR}_1$, which seems advantageous in light of Lemma~\ref{lem:evcomp2}.
\end{remark}

\begin{remark}\label{rem:rem2}
As in Remark~\ref{rem:constweightswobd} we may consider the case where $\al=\bar{\al}/|\Om|$, $\be=(1-\bar{\al})/|\dOm|$ for $\bar{\al},\bar{\be}\in(0,1)$ constant, i.e.\ constant weight functions $\al,\be$. This corresponds to the following infinitesimal generator
\begin{equation*}
\bar{L}f := \1_\Om \Delta f + \1_{\dOm} \lp \Delta^\tau f  - \gamma \partial_N f \rp, 
\end{equation*}
where $\gamma:=\frac{\bar{\al}}{1-\bar{\al}} \frac{|\dOm|}{|\Om|}$. For growing $\gamma$ the effect of the reflection from the boundary becomes stronger, while the effect of the stickiness and boundary diffusion loses relevance, and the diffusion thus increasingly resembles usual reflected Brownian motion without stickiness and boundary diffusion.\\
For these constant weight functions
\begin{equation*}
\overline{h}_D = \inf_{\substack{A,B,C\\|\cdot|_\al+|\partial_E \cdot|_\be\le 1/2}} \frac{ |\partial_I A|}{|A|}  \frac{|\partial_I B|}{|B| + \frac{1}{\gamma}|\partial_E B|}
+ \frac{\frac{1}{\gamma}|\partial\partial_E B|}{|B| + \frac{1}{\gamma}|\partial_E B|} \frac{|\partial \partial_E C|}{|\partial_E C|}.
\end{equation*}
If we exchange infimum and limit this tends to $\overline{h}_C^2$ as $\gamma$ goes to $\infty$, the same holds for $\overline{h}$ replaced by $\tilde{h}$ or $h$. Thus, up to our adjusted infimal condition, we again reproduce the classical Cheeger inequality in the limit. Note again that we did not prove $\Gamma$-convergence rigorously.\\
Alternatively we may also consider $\delta>0$ instead of just $\delta=1$ as above. The parameter $\delta$ controls the speed of the boundary diffusion. The corresponding generator is
\begin{equation*}
Lf := \1_\Om \lp \Delta f + \frac{\nabla \al}{\al}\cdot \nabla f \rp - \1_{\dOm} \frac{\al}{\be} \partial_N f + \1_{\dOm} \delta\lp \Delta^\tau f + \frac{\nabla^\tau \be}{\be}\cdot \nabla^\tau f\rp.
\end{equation*}
Proceeding as in the proof of Theorem~\ref{thm:cheegerwbd} we then obtain analogous lower bounds for the first eigenvalue in terms of the following Cheeger-type constants
\begin{gather*}
\tilde{h}_D := \inf_{\substack{A,B,C\\|\partial_E \cdot|_\be\le |\dOm|_\be/2}} \frac{|\partial_I A|_\al}{|A|_\al} \frac{|\partial_I B|_\al}{|B|_\al +|\partial_E B|_\be} + 
 \delta\cdot \frac{|\partial\partial_E B|_\be}{|B|_\al  + |\partial_E B|_\be} \frac{|\partial \partial_E C|_\be}{|\partial_E C|_\be},\\
\tilde{h}_E := \inf_{|\partial_E B|_\be \le |\dOm|/2}  \frac{ |\partial_I B|_\al + \delta\cdot|\partial \partial_E B|_\be}{ |B|_\al + |\partial_E B|_\be}.
\end{gather*}
Now considering $\delta\to 0$ slows down the boundary diffusion while keeping up the sticky reflection. In fact $\tilde{h}_D$ tends to $\tilde{h}_B\cdot \tilde{h}_C$ as $\delta$ tends to zero (again we exchange infimum and limit), and we thus recover the lower bound from the previous section for the case of sticky reflection without boundary diffusion.
\end{remark}

\begin{remark}
We may again ask whether/in which way dumbbell examples are extremal for $h_D$. We therefor rewrite as follows
\begin{align*}
	&\frac{|\partial_I A|_\al }{|A|_\al} \frac{ |\partial_I B|_\al}{ |B|_\al + |\partial_E B|_\be}
	+  \frac{|\partial\partial_E B|_\be}{|B|_\al + |\partial_E B|_\be}\frac{|\partial \partial_E C|_\be}{|\partial_E C|_\be}\\
	\le & \frac{ |\partial_I A|_\al |\partial_I B|_\al}{|A|_\al |B|_\al} + \frac{|\partial\partial_E B|_\be|\partial \partial_E C|_\be}{|\partial_E B|_\be|\partial_E C|_\be}.
\end{align*}
Then it is obvious that we may construct a sequence of dumbbell examples such that the corresponding values of $h_D$ converge to zero: Such a sequence of dumbbell manifolds $\Omega_\epsilon$ can be constructed e.g.\ by taking two copies of a unit ball $B_1^n$ in $\mathbb{R}^n$ with standard Euclidean metric $g_n$, removing a disk of radius $\epsilon$ from the boundary of each, connecting them by $[-L_\ep/2,L_\ep/2]\times B_\epsilon^{n-1}$ ($ B_\epsilon^{n-1}$ denotes the Euclidean ball of radius $\epsilon$ in dimension $n-1$ and $L_\ep$ can be chosen constant or increasing as $\ep$ tends to zero) with product metric $g=dx^2+g_{n-1}$ and smoothing the metric near the junctions. The quotient written above goes to zero as $\epsilon$ tends to zero by choosing $A_\epsilon, B_\epsilon, C_\epsilon$ as either half of the dumbbell $\Omega_\epsilon$.\\
When considering the rate of convergence to equilibrium of Brownian motion with sticky-reflecting boundary diffusion both diffusion in the interior and diffusion on the boundary are responsible for the overall spreading of the process on the domain. Thus a domain has a particularly low spectral gap if both its interior and its boundary are opposing a fast spreading of the process (as for the dumbbell) or the domain is such that stickiness at the boundary dominates the behaviour. \\
We have shown above that for Brownian motion reflected at the boundary (Neumann boundary condition) adding stickiness worsens the rate of convergence to equilibrium and adding diffusion along the boundary on top of that improves it. However, we have not compared the rate of convergence to equilibrium between the cases of pure reflection and sticky-reflecting boundary diffusion. To conclude, we present an example which shows that by making the diffusion along the boundary sufficiently fast the spectral gap for the case of sticky-reflecting boundary diffusion can be larger than in the case of pure reflection. It remains open to study and quantify this in general.

\begin{example} 
We consider again the example of a closed unit ball in the Euclidean plane, but this time with constant weight functions $\al=5/7\pi, \be=1/7\pi$ (corresponding to $\gamma=5$) and varying boundary diffusion speed $\delta$. As in~\cite{mvr} we may calculate the exact value of $\la_1^{SRBD}$ for varying $\delta$ numerically and compare it with $\la_1^N\approx 3.39$ for this example.
We plot the resulting curves in Figure~\ref{fig:euclidex2}. Note that indeed for $\delta$ large enough $\la^{SRBD}_1$ becomes larger than $\la_1^N$. Of course $\la^{SRBD}_1$ will not grow arbitrarily large with growing $\delta$.
\end{example}
\begin{figure}
	\includegraphics[width=0.9\linewidth]{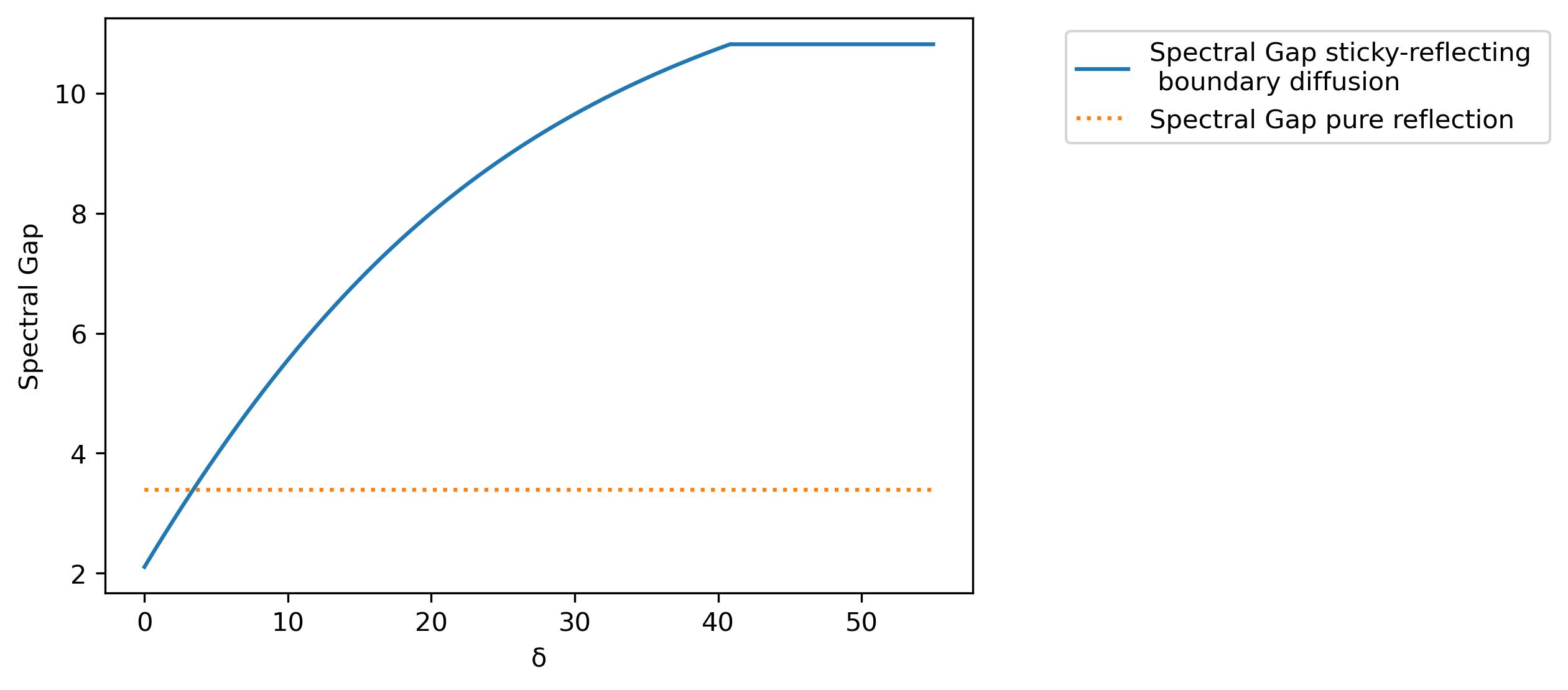}
	\caption{Spectral gap for sticky-reflecting boundary diffusion of varying speed $\delta$ (blue, solid) and spectral gap for pure reflection (yellow, dotted) for Euclidean example}
	\label{fig:euclidex2}
\end{figure}

\end{remark}

\section{Towards Buser-type inequalities}

As a counterpart to Cheeger's inequality Buser showed in~\cite{MR683635} that for a compact Riemannian manifold $\Om$ without boundary if $Ric \ge -(n-1)\delta^2, \delta\ge0$, then it holds 
\begin{equation*}
\lambda_1 \le c_1 (\delta h_C + h_C^2),
\end{equation*}
where $c_1$ is a constant which depends only on the dimension. A simple analytic proof for this result can also be found in~\cite{MR1186991,MR2195409}. Furthermore, versions of such Buser inequalities bounding the spectral gap from above in terms of Cheeger-constants have also been shown for the cases of nonempty boundary with Neumann or Dirichlet boundary condition in~\cite{MR2040788,MR3881014}. Meanwhile, for the Steklov problem it has been argued in~\cite{MR4695859} that a Buser-type inequality bounding $\lambda_1^S$ from above in terms of $h_J$ and $h_C$ cannot hold in general without further constraints.\\
In this last section we make a step towards showing Buser-type inequalities for the first non-trivial eigenvalue of the drift Laplacian on manifolds with Wentzell-type boundary condition (both for $\delta=0$ and $\delta>0$) by showing that suitable heat kernel estimates imply Buser-type inequalities. The proofs are analogous to~\cite{MR1186991}.

\begin{theorem}\label{thm:buser1}
Let $(\hat{P}_t)_t$ be the semigroup associated to Brownian motion with sticky reflection but without boundary diffusion, doubly weighted according to $\mu$. Assume that for some $t_0>0$ and some constant $C$
\begin{equation*}
|\ |\nabla \hat{P}_t f|\ |_{\infty,\la} \le \frac{C}{\sqrt{t}} |f|_{\infty,\la}\ \forall 0< t \le t_0,
\end{equation*}
where $|\cdot|_{\infty,\la}$ is the $L^\infty$-norm on $\Om$ with respect to the volume measure $\la$. Then it holds
\begin{equation*}
\lambda_1^{SR} \le 12C^2 \lp \overline{h}_B\rp^2 + \frac{4C}{\sqrt{t_0}} \overline{h}_B.
\end{equation*}
\end{theorem}

\begin{proof}
Denote by $|\cdot|_{1,\mu},|\cdot|_{1,\al\la}$ the $L^1$-norm on $\Om$ with respet to $\mu$ resp.\ $\la$ weighted by $\al$. It holds 
\begin{equation}\label{eq:l1ineqwobd}
|f-\hat{P}_tf|_{1,\mu} \le  2 C \sqrt{t} |\nabla f|_{1,\al\la}
\end{equation}
for every $f\in \dD(\hat{L}_\mu)$ since for every smooth $g$ with $|g|_{\infty,\mu}\le1$
\begin{align*}
\int_\Om g(f-\hat{P}_tf) d\mu =& - \int_\Om \int_0^t \frac{d}{ds} g\hat{P}_s f ds d\mu=- \int_0^t \int_\Om g\hat{L}_\mu \hat{P}_s f d\mu ds \\
	= &- \int_0^t \int_\Om \hat{P}_sg \hat{L}_\mu f d\mu ds \\
	= &- \int_0^t -\int_\Om \nabla \hat{P}_sg \nabla f \al d\la + \int_{\dOm} \hat{P}_s g \frac{\partial f}{\partial N} \al d\la  - \int_{\dOm} \hat{P}_sg \frac{\partial f}{\partial N} \al d\si ds\\
	= & \int_0^t \int_\Om \nabla \hat{P}_sg \nabla f \al d\la  ds \\
	\le & |\nabla f|_{1,\al\la} \int_0^t |\nabla \hat{P}_s g|_{\infty,\la} ds\\
	\le  & |\nabla f|_{1,\al\la} \cdot 2C\sqrt{t}|g|_{\infty,\la} \\
	\le& 2C\sqrt{t} |\nabla f|_{1,\al\la} .
\end{align*}
Here
\begin{equation*}
\hat{L}_\mu f:=  \lp\Delta f +\frac{\nabla \al \cdot \nabla f}{\al} \rp\1_{\Om^\circ} - \lp \frac{\al}{\be} \frac{\partial f}{\partial N} \rp\1_{\dOm}
\end{equation*}
is the infinitesimal generator of $(\hat{P}_t)_t.$
For a smooth open set $M$ in $\Om$ with smooth boundary $\partial M$ we now apply inequality \eqref{eq:l1ineqwobd} to smooth functions in $\dD(\hat{L}_\mu)$ approximating $\1_M$ to obtain
\begin{align*}
2C\sqrt{t} |\partial_I M|_\al &\ge \int_M \1_M - \hat{P}_t(\1_M) d\mu + \int_{M^C} \hat{P}_t(\1_M) d\mu \\
	&=  |M|_\al + |\partial_E M|_\be - \int_M \hat{P}_t(\1_M) d\mu + \int_{M^C} \hat{P}_t(\1_M) d\mu \\
	&= 2\lp |M|_\al + |\partial_E M|_\be - \int_M \hat{P}_t(\1_M) d\mu \rp \\
	&= 2\lp |M|_\al + |\partial_E M|_\be -|\hat{P}_{t/2}\1_M|^2_{2,\mu} \rp.
\end{align*}
Now using that for $f$ smooth with $\int_\Om f d\mu=0$
\begin{equation*}
|\hat{P}_t f|^2_{2,\mu} \le e^{-2t\la_1^{SR}} |f|^2_{2,\mu}, \ t\ge0,
\end{equation*}
we obtain
\begin{align*}
|\hat{P}_{t/2}\1_M|^2_{2,\mu} &= \lp|M|_\al+|\partial_E M|_\be\rp^2 + |\hat{P}_{t/2}(\1_M - |M|_\al - |\partial_E M|_\be) |^2_{2,\mu} \\
	&\le  \lp|M|_\al+|\partial_E M|_\be\rp^2 + e^{-t\la_1^{SR}} |\1_M - |M|_\al - |\partial_E M|_\be|^2_{2,\mu}.
\end{align*}
Thus,
\begin{align*}
&2C\sqrt{t} |\partial_I M|_\al \\
	\ge& 2\lp |M|_\al + |\partial_E M|_\be -|\hat{P}_{t/2}\1_M|^2_{2,\mu} \rp \\
	\ge& 2\lp |M|_\al + |\partial_E M|_\be - \lp|M|_\al+|\partial_E M|_\be\rp^2 - e^{-t\la_1^{SR}} |\1_M - |M|_\al - |\partial_E M|_\be|^2_{2,\mu} \rp \\
	=& 2\lp |M|_\al + |\partial_E M|_\be - \lp|M|_\al+|\partial_E M|_\be\rp^2 - e^{-t\la_1^{SR}} \lp  |M|_\al + |\partial_E M|_\be - \lp  |M|_\al + |\partial_E M|_\be \rp^2 \rp\rp \\
	=& 2\lp |M|_\al + |\partial_E M|_\be\rp \lp 1- |M|_\al - |\partial_E M|_\be\rp \lp 1 - e^{-t\la_1^{SR} }\rp.
\end{align*}
This implies
\begin{equation*}
\overline{h}_B = \inf_{|M|_\al + |\partial_E M|_\be \le \frac{1}{2}} \frac{|\partial_I M|_\al}{|M|_\al + |\partial_E M|_\be} \ge \sup_{t_0\ge t>0}\frac{1}{2C} \frac{1-e^{-t\la_1^{SR}}}{\sqrt{t}}.
\end{equation*}
If $\lambda_1^{SR}t_0\ge 1$, we choose $t=1/\lambda_1^{SR}$ to obtain
\begin{equation*}
\overline{h}_B \ge \frac{1}{2C} (1-e^{-1})\sqrt{\lambda_1^{SR}}.
\end{equation*}
Else $\lambda_1^{SR} t_0< 1$ and we choose $t=t_0$ to obtain 
\begin{equation*}
\overline{h}_B \ge \frac{1}{2C} \frac{1-e^{-t_0 \lambda_1^{SR}}}{\sqrt{t_0}} \ge \frac{\sqrt{t_0}}{4C} \lambda_1^{SR},
\end{equation*}
where we used that $(1-e^{-x})\ge x/2$ for $0\le x \le 1$. Thus, overall
\begin{equation*}
\lambda_1^{SR} \le 12C^2 \lp \overline{h}_B\rp^2 + \frac{4C}{\sqrt{t_0}} \overline{h}_B.
\end{equation*}
\end{proof}

Furthermore, we may show an analogous statement in the case $\delta=1$, i.e.\ including diffusion along the boundary: 

\begin{theorem}\label{thm:buser2}
Let $(P_t)_t$ be the semigroup associated to Brownian motion with sticky-reflecting boundary diffusion, doubly weighted according to $\mu$. Assume that for some $t_0>0$ and some constant $D$
\begin{equation*}
|\ |\nabla P_t f|\ |_{\infty,\la} \le \frac{D}{\sqrt{t}} |f|_{\infty,\la} \text{ and } |\ |\nabla^\tau P_t f|\ |_{\infty,\si} \le \frac{D}{\sqrt{t}} |f|_{\infty,\si} \ \forall 0< t \le t_0,
\end{equation*}
where $|\cdot|_{\infty,\la}$ is the $L^\infty$-norm on $\Om$ with respect to the volume measure $\la$ and $|\cdot|_{\infty,\si}$ is the $L^\infty$-norm on $\dOm$ with respect to the measure $\si$. 
Then it holds 
\begin{equation*}
\lambda_1^{SRBD} \le 12D^2 \lp \overline{h}_E\rp^2 + \frac{4D}{\sqrt{t_0}} \overline{h}_E.
\end{equation*}
\end{theorem}
\begin{proof}
In addition to the notation introduced in the previous proof denote by $|\cdot|_{1,\be\si}$ the $L^1$-norm on $\dOm$ with respet to $\si$ weighted by $\be$. It holds 
\begin{equation}\label{eq:l1ineqwbd}
|f-P_tf|_{1,\mu} \le  2C\sqrt{t}\lp |\nabla f|_{1,\al\la}  + |\nabla^\tau f|_{1,\be\si} \rp,
\end{equation}
for every $f\in\dD(L_\mu)$ since for every smooth $g$ with $|g|_{\infty,\mu}\le1$
\begin{align*}
\int_\Om g(f-P_tf) d\mu =& - \int_\Om \int_0^t \frac{d}{ds} gP_s f d\mu=- \int_0^t \int_\Om gL_\mu P_s f d\mu ds \\
	= &- \int_0^t \int_\Om P_sg L_\mu f d\mu ds \\
	= &- \int_0^t -\int_\Om \nabla P_sg \nabla f \al d\la + \int_{\dOm} P_s g \frac{\partial f}{\partial N} \al d\la \\
	&- \int_{\dOm} \nabla^\tau P_sg \nabla^\tau f \be d\si  - \int_{\dOm} P_sg \frac{\partial f}{\partial N} \al d\si ds\\
	= & \int_0^t \int_\Om \nabla P_sg \nabla f \al d\la  + \int_{\dOm} \nabla^\tau P_sg \nabla^\tau f \be d\si  ds \\
	\le & |\nabla f|_{1,\al\la} \int_0^t |\nabla P_s g|_{\infty,\la} ds +  |\nabla^\tau f|_{1,\be\si} \int_0^t |\nabla^\tau P_s g|_{\infty,\si} ds \\
	\le  & |\nabla f|_{1,\al\la} \cdot 2D\sqrt{t}|g|_{\infty,\la} +  |\nabla^\tau f|_{1,\be\si} \cdot 2D \sqrt{t}|g|_{\infty,\si}\\
	\le& 2D\sqrt{t}\lp |\nabla f|_{1,\al\la}  + |\nabla^\tau f|_{1,\be\si} \rp.
\end{align*}
Here 
\begin{equation*}
L_\mu f:= \lp\Delta f +\frac{\nabla \al \cdot \nabla f}{\al} \rp \1_{\Om^\circ} - \lp\frac{\al}{\be} \frac{\partial f}{\partial N} \rp \1_{\dOm}+\lp \Delta^\tau f + \frac{\nabla^\tau \be \cdot \nabla^\tau f}{\be} \rp \1_{\dOm}
\end{equation*}
is the infinitesimal generator of $(P_t)_t$.
For a smooth open set $M$ in $\Om$ with smooth boundary $\partial M$ we now apply inequality \eqref{eq:l1ineqwbd} to smooth functions in $\dD(L_\mu)$ approximating $\1_M$ to obtain
\begin{align*}
&2D\sqrt{t}\lp |\partial_I M|_\al+|\partial \partial_E M|_\be \rp \\
	&\ge \int_M \1_M - P_t(\1_M) d\mu + \int_{M^C} P_t(\1_M) d\mu \\
	&=  |M|_\al + |\partial_E M|_\be - \int_M P_t(\1_M) d\mu + \int_{M^C} P_t(\1_M) d\mu \\
	&= 2\lp |M|_\al + |\partial_E M|_\be - \int_M P_t(\1_M) d\mu \rp \\
	&= 2\lp |M|_\al + |\partial_E M|_\be -|P_{t/2}\1_M|^2_{2,\mu} \rp.
\end{align*}
Now using that for $f$ smooth with $\int_\Om f d\mu=0$
\begin{equation*}
|P_t f|^2_{2,\mu} \le e^{-2t\la_1^{SRBD} } |f|^2_{2,\mu}, \ t\ge0,
\end{equation*}
we obtain
\begin{align*}
|P_{t/2}\1_M|^2_{2,\mu} &= \lp|M|_\al+|\partial_E M|_\be\rp^2 + |P_{t/2}(\1_M - |M|_\al - |\partial_E M|_\be) |^2_{2,\mu} \\
	&\le  \lp|M|_\al+|\partial_E M|_\be\rp^2 + e^{-t\la_1^{SRBD}} |\1_M - |M|_\al - |\partial_E M|_\be|^2_{2,\mu}.
\end{align*}
Thus,
\begin{align*}
&2D\sqrt{t}\lp |\partial_I M|_\al+|\partial \partial_E M|_\be \rp \\
	\ge& 2\lp |M|_\al + |\partial_E M|_\be -|P_{t/2}\1_M|^2_{2,\mu} \rp \\
	\ge& 2\lp |M|_\al + |\partial_E M|_\be - \lp|M|_\al+|\partial_E M|_\be\rp^2 - e^{-t\la_1^{SRBD}} |\1_M - |M|_\al - |\partial_E M|_\be|^2_{2,\mu} \rp \\
	=& 2\lp |M|_\al + |\partial_E M|_\be - \lp|M|_\al+|\partial_E M|_\be\rp^2 - e^{-t\la_1^{SRBD}} \lp  |M|_\al + |\partial_E M|_\be - \lp  |M|_\al + |\partial_E M|_\be \rp^2 \rp\rp \\
	=& 2\lp |M|_\al + |\partial_E M|_\be\rp \lp 1- |M|_\al - |\partial_E M|_\be\rp \lp 1 - e^{-t\la_1^{SRBD}}\rp.
\end{align*}
This implies
\begin{equation*}
\overline{h}_E = \inf_{|M|_\al +|\partial_E M|_\be \le \frac{1}{2}} \frac{|\partial_I M|_\al+|\partial \partial_E M|_\be}{|M|_\al + |\partial_E M|_\be} \ge \sup_{t>0}\frac{1}{2D} \frac{1-e^{-t\la_1^{SRBD}}}{\sqrt{t}}.
\end{equation*}
If $\lambda_1^{SRBD}t_0\ge 1$, we choose $t=1/\lambda_1^{SRBD}$ to obtain
\begin{equation*}
\overline{h}_E \ge \frac{1}{2D} (1-e^{-1})\sqrt{\lambda_1^{SRBD}}.
\end{equation*}
Else $\lambda_1^{SRBD} t_0< 1$ and we choose $t=t_0$ to obtain 
\begin{equation*}
\overline{h}_E \ge \frac{1}{2D} \frac{1-e^{-t_0 \lambda_1^{SRBD}}}{\sqrt{t_0}} \ge \frac{\sqrt{t_0}}{4D} \lambda_1^{SRBD},
\end{equation*}
where we used that $(1-e^{-x})\ge x/2$ for $0\le x \le 1$. Thus, overall
\begin{equation*}
\lambda_1^{SRBD} \le 12D^2 \lp \overline{h}_E\rp^2 + \frac{4D}{\sqrt{t_0}} \overline{h}_E.
\end{equation*}
\end{proof}

\begin{remark}\label{rem:someotherrem}
Note that analogously we may show
\begin{equation*}
\lambda_1^{SR} \le \frac{12C^2}{B^2} \lp h_B\rp^2 + \frac{4C}{\sqrt{t_0}B} h_B\ \text{ or }\ \lambda_1^{SR} \le \frac{12C^2}{A^2} \lp \tilde{h}_B\rp^2 + \frac{4C}{\sqrt{t_0}A} \tilde{h}_B,
\end{equation*}
and 
\begin{equation*}
\lambda_1^{SRBD} \le \frac{12D^2}{B^2} \lp h_E\rp^2 + \frac{4D}{\sqrt{t_0}B} h_E\ \text{ or }\ \lambda_1^{SRBD} \le \frac{12D^2}{A^2} \lp \tilde{h}_E\rp^2 + \frac{4D}{\sqrt{t_0}A} \tilde{h}_E,
\end{equation*}
where $A:=\int_\Om \al d\la$, $B:=\int_{\dOm} \be d\si$, $A,B\in (0,1)$.\\ 
Also for $0<\delta\neq 1$ we get analogous results in terms of the heat kernel estimate for the semigroup corresponding to $\delta$ and the adapted Cheeger-type constant
\begin{equation*}
\overline{h}_E = \inf_{|M|_\al +|\partial_E M|_\be \le \frac{1}{2}} \frac{|\partial_I M|_\al+\delta|\partial \partial_E M|_\be}{|M|_\al + |\partial_E M|_\be}
\end{equation*} 
(or the adapted $h_E$ resp.\ $\tilde{h}_E$).
\end{remark}

In order to obtain full Buser-type inequalities it remains to show heat kernel estimates as assumed in Theorems~\ref{thm:buser1} and~\ref{thm:buser2}. Such heat kernel estimates may be checked via explicit computations in simple settings where the transition kernel for Brownian motion with sticky-reflecting boundary diffusion has been computed explicitly, see e.g.~\cite{casteras2025largedeviationsstickyreflectingbrownian}. However, at present we are not aware of any result providing such estimates in the general setting of this manuscript. \\
We expect constants $C$ and $D$ in the assumptions of Theorems~\ref{thm:buser1} resp.~\ref{thm:buser2} to depend on $t_0$ as well as the curvature of $\Om$ and $\dOm$. Indeed, this would be analogous to the existing results for more standard boundary behaviour: In case of a non-empty convex boundary and the semigroup associated with the Neumann or Dirichlet Laplacian the constant $C$ may be chosen depending on $t_0$, a lower bound on the Ricci curvature and the dimension $d$, cf.~\cite{MR3881014}. The case of a nonconvex boundary is also treated in~\cite{MR3881014}. Thus, the Buser inequality would give an upper bound on the first non-trivial eigenvalue of the drift Laplacian with Wentzell-type boundary condition ($\delta=0$ or $\delta>0$) in terms of the Cheeger-type constants and in terms of lower bounds on Ricci curvature and second fundamental form on $\dOm$.\\

\noindent\textbf{Acknowledgements} The author would like to thank Max von Renesse for suggesting the problem and Sebastian
Boldt for helpful discussions. The author would also like to thank the anonymous reviewer for many helpful comments.\\

\noindent\textbf{Funding} Open Access funding enabled and organized by Projekt DEAL.\\

\noindent\textbf{Data Availaility Statement} No datasets were generated or analysed during the current study.

\section*{Declarations}

\noindent\textbf{Conflicts of Interest} The author has no relevant financial or non-financial interests to disclose.

\bibliographystyle{abbrv}
\bibliography{bib}
\end{document}